\documentclass[12pt]{article}
\usepackage{amsmath, amssymb, amsthm}
\usepackage{enumerate}

\newtheorem*{theorema}{Proposition 1}
\newtheorem*{theoremb}{Proposition 2}
\newcommand{\R}{{\mathbb{R}}}
\newcommand{\taf}{{\hskip 5pt} $\blacksquare$
                  \renewcommand{\qedsymbol}{}}
\begin{document}
\title{A characterization of the Lorentz Space\\ $L(p,r)$ in terms of Orlicz type classes}
\author{Calixto P. Calder\'on and Alberto Torchinsky}
\date{}
\maketitle

{\em In remembrance of N. M. Rivi\`ere (1940 -- 1978), who believed in \\ Lorentz Spaces.} 

\begin{abstract}  We describe the Lorentz space $L(p,r)$, $0<r<p$, $p>1$, in terms of Orlicz type classes of functions $L_{\Psi}$. As a consequence of this result it follows that Stein's characterization of the real functions on $\R^n$ that are differentiable at almost all the points in $\R^n$ \cite{Stein}, is equivalent to the earlier characterization of those functions given by A. P. Calder\'on \cite{APC}.  
\end{abstract}

\section*{Introduction}
 
In 1981 E. M. Stein  proved that if  the gradient $\nabla F$ in the distribution sense of  a real function  $F$ on $\R^n$  belongs to the Lorentz space $L(n,1)$, then $F$ is differentiable at almost all the points in $\R^n$, $n>1$. He further proved that no condition on $\nabla F$ weaker than $\| \nabla F\|_{n,1}^*<\infty$ will guarantee the differentiability  of $F$ a.e.\! in $\R^n$,  \cite{Stein}. E. M. Stein refers to ``local" $L(n,1)$, nevertheless the connection between ``local" and ``global" should be clear to the reader in this context.

Earlier, in 1951, A. P. Calder\'on had proved that if $\nabla  F$ belongs to the Orlicz class
\begin{equation}
L_{\Psi}=\Big\{f: \int_B \Psi\big(|f(x)|\big) \,dx<\infty\Big\},
\end{equation}
$B$ a ball in $\R^n$ and $\Psi$ satisfying 
\begin{equation}
\int_1^\infty \big(t/\Psi(t)\big)^{1/(n-1)}\,dt<\infty,
\end{equation}
$\Psi(t)$ nonnegative, nondecreasing, then $F$ is differentiable   
at almost all the points of $B$. A. P. Calder\'on further showed that no condition on $\nabla F$ weaker than (1), (2) above guarantees the a.e.\!\! differentiablity on $B$ \cite{APC}.
Since Calder\'on's proof suggests that convexity may not be necessary for $\Psi$, we will not require it in what follows. We will refer to those classes of functions   as Orlicz type classes.

The aim of this paper is to establish the connection between the Lorentz space $L(p,r)$,  $0<r<p$, $p>1$, and Orlicz type classes that satisfy a condition akin to (2) above. The case $p=n$, $r=1$, is of particular interest as it    implies that the  differentiability conditions  discussed above are equivalent \cite{CPC0}. 

In a related context, since both  Lorentz and Orlicz spaces, as well as the hybrid Lorentz -- Orlicz spaces, arise as intermediate spaces of   $L^p$ spaces [9], it is also of interest to describe their interconnections.  Now, if $I=[0,1]$ denotes the unit interval in $\R$,  $L^p(I)$ cannot be expressed as the union of the $L^q(I)$ spaces  it contains properly; 
$f(x)= |x|^{-1/p} \, \ln^{-2/p}(1/|x)|)\chi_{I}\in L^p(I)$ and $f\notin L^{q}(I)$ for $p<q\le\infty $. On the other hand, 
Welland showed that $L^p(I)$, as well as more general  Orlicz spaces on $I$, can be represented as the union of Orlicz spaces that they contain properly \cite{RW}.   
Since the Lorentz spaces are monotone with respect to the second index \cite{RH}, and since
 $L(p,r)(I)\subset L(p,p)(I)=L^p(I)$ in the range that is of interest to us, Welland's result gives that $L(p,r)(I)$  can be described as the union of Orlicz spaces that it contains, but this is insufficient to us. Our result covers  $L(p,r)(\R^n)$  and, more to the point,  the Orlicz type classes $L_\Psi$   are taken over the family of  functions $\Psi$ that satisfy  condition (3) below.  
\vskip 3pt

We will   work with classes of measurable functions $f$ defined on $\R^n$. Let  $f^*$ denote the nonincreasing rearrangement of $|f|$, and let $L(p,r)$,  \cite{RH,NMR},  denote the Lorentz space of measurable functions $f$ whose nonincreasing rearrangement $f^*$ satisfies
\begin{equation*}
\int_0^\infty f^*(t)^r\, t^{r/p-1}\, dt<\infty.
\end{equation*}
We will restrict ourselves to  the range $0<r<p$, $p>1$.

Also consider the Orlicz type class $L_\Psi$ of measurable functions $f$ defined on $\R^n$, such that the rearrangement $f^*$  of $|f|$ satisfies
\begin{equation*}
\int_0^\infty \Psi\big( f^*(t)\big)\,dt<\infty,
\end{equation*}
for a nondecreasing $\Psi\ge 0$ defined on $(0,\infty)$ and satisfying
\begin{equation}
\int_0^\infty \frac{t^{q-1}}{\Psi(t)^{q/p}}\,dt<\infty,
\end{equation}
where $0<r<p$, $p>1$, and $1/p+1/q=1/r$.

The aim of this paper is to prove that for $0<r<p$, $p>1$, 
\begin{equation}
L(p,r)=\bigcup_\Psi L_\Psi,\quad 1/p+1/q=1/r.
\end{equation}
Or, in other words,  $f\in L(p,r)$ if and only if 
\begin{equation*}
\int_0^\infty \Psi(f^*(t) )\,dt<\infty
\end{equation*}
for some $\Psi$  that satisfies (3) above.

The proof is accomplished in two parts, each dealing with an inclusion in (4). We only point out that the  constants $c$ that appear below may vary from appearance to appearance, and are independent of $f$.

\section*{Embedding of Orlicz  type classes into Lorentz Spaces}

We begin by showing  that the Orlicz  type classes corresponding to functions $\Psi$ that satisfy (3) above are continuously included in an appropriate Lorentz space. More precisely, we have 
\begin{theorema}
Let $f$ be a nonnegative, nonincreasing function  defined on $(0,\infty)$ such that 
\begin{equation*}
\int_0^\infty \Psi\big(f(t)\big)\,dt<\infty, 
\end{equation*}
where  $\Psi(t)$ satisfies (3) above. 

Then, we have 
\begin{equation}
\int_0^\infty f(t)^r \, t^{r/p-1}\,dt
\le c\Big(\int_0^\infty \frac{t^{q-1}}{\Psi(t)^{q/p}}\,dt\Big)^{r/q}\Big(\int_0^\infty 
\Psi\big(f(t)\big)\,dt\Big)^{r/p}.
\end{equation}
\end{theorema}

\begin{proof}
Let 
\begin{equation*} J= \int_0^\infty f(t)^r \, t^{r/p-1}\,dt,
\end{equation*}
and consider the interval  $I_k$   where $2^k<f\le 2^{k+1}$, ${-\infty<k<\infty}$. Clearly 
\begin{equation}
J \le \sum_{k} 2^{(k+1)r} \int_{I_k}t^{r/p-1}\,dt\le (p/r)\sum_k 2^{(k+1)r}\,|I_k|^{r/p}\,.
\end{equation}

Now, multiplying and dividing by $\Psi(2^k)^{r/p}$, by H\"older's inequality with conjugate 
 indices $(p/r, q/r)$, it readily follows that the sum in the right-hand side of (6)  is  dominated by
\begin{align}
\sum_k 2^{(k+1)r} \Psi(2^k)^{-r/p} &|I_k|^{r/p}\, \Psi(2^k)^{r/p}\nonumber\\
&\le  \Big(\sum_k \frac{2^{(k+1)q}}{\Psi(2^k)^{q/p}} \Big)^{r/q} \Big( \sum_k |I_k|\,\Psi(2^k)\Big)^{r/p}.
\end{align}

Consider the sum in the first factor in (7) above.  Each summand there  can be estimated by
\begin{equation*}
 \frac{2^{(k+1)q}}{\Psi(2^k)^{q/p}}\le c  \int_{2^{k-1}}^{2^{k}} \frac{t^{q-1}}{\Psi(t)^{q/p}}\,dt\,,
\end{equation*}
and, consequently, the sum does not exceed
\begin{equation}
c\,\sum_{k} \int_{2^{k-1}}^{2^{k}}\frac{t^{q-1}}{\Psi(t)^{q/p}}\,dt
=c \int_0^\infty \frac{t^{q-1}}{\Psi(t)^{q/p}}\,dt\,. 
\end{equation}

As for the second sum, since 
\[ |I_k|\,\Psi(2^k) \le \int_{I_k} \Psi\big(f(t)\big)\,dt\,,
\]
it readily follows that
 \begin{equation}
 \sum_k |I_k|\,\Psi(2^k) \le \int_0^\infty \Psi \big(f(t)\big)\,dt\,.
\end{equation}
Thus, combining (6), (7), (8), and (9) above, the estimate (5) holds, and the proof is finished.
\taf\end{proof}

\section*{Embedding of Lorentz Spaces into Orlicz  type classes} We complete the proof by showing that if $f\in L(p,r)$ for an appropriate range of values of $p,r$, then $f$ is in an  Orlicz type class $L_\Psi$, where $\Psi$ depends on $f$. 
More precisely, we have
\begin{theoremb}
Let $0<r<p$, $p>1$. Let $f$ be a 
 nonnegative, nonincreasing  function defined on $(0,\infty)$ such that 
 \begin{equation*}
 \int_0^\infty f(t)^r\, t^{r/p-1}\,dt<\infty\,. 
 \end{equation*} 
Then, with $1/p+1/q=1/r$, there exists  a nonnegative, nondecreasing  function $\Psi(t)$ defined on 
$(0,\infty)$  satisfying (3) above  
for which
\begin{equation} \int_0^\infty \Psi\big(f(t)\big)\,dt<\infty.
\end{equation}
\end{theoremb}

\begin{proof}
Let 
 $f_0(t)$ be a  strictly positive, strictly decreasing function on $(0,\infty)$ such that
\[ \lim_{t\to 0^+} f_0(t)=\infty,\qquad  \lim_{ t\to\infty} f_0(t)= 0,\]
and
\[\int_0^\infty f_0(t)^r \, t^{r/p-1}\, dt< \int_0^\infty f(t)^r \, t^{r/p-1}\, dt\,. 
\]

Let now $g_0=f+f_0$. Then, 
\begin{equation}
 f(t)< g_0(t),\quad {\text{ all }}t,
\end{equation}
and 
\begin{equation}
\int_0^\infty g_0(t)^r\, t^{r/p-1}\,dt\le {\text {max }}(2,2^{r})  \int_0^\infty f(t)^r\,t^{r/p-1}\,dt <\infty\,.
\end{equation}

Finally, we define the function $g(t)$. Let $J_k$ be the interval where $2^k<g_0(t)\le 2^{k+1}$, and $[a_k,b_k]$ its closure. Then $g(t)$ is defined by
\begin{equation*}
g(a_k)= 2^{k+1},\qquad  g(b_k)= 2^k 
 \end{equation*} 
 and extended  linearly on $[a_k,b_k]$. It  follows that
 $g(t)$ is strictly decreasing, continuous on $(0,\infty)$,  absolutely continuous, invertible,  and, 
 \begin{equation}
 g_0(t)/2<g(t)< 2\, g_0(t)\,.
 \end{equation}
Furthermore, since $g(t)$ is  decreasing and $r<p$ it follows that
\[   g(\varepsilon)^r\, \varepsilon^{r/p}\le\int_0^\varepsilon  g(t)^r\, t^{r/p-1}\, dt,\quad \varepsilon>0,
\]
and, consequently,
\begin{equation}\lim_{\varepsilon\to 0^+}  g(\varepsilon)^r\, \varepsilon^{r/p}=0.
\end{equation}
 
Likewise, for large $N$,  we have 
\begin{equation*}
g(N)^r\, N^{r/p}\le c\,\int_{N/2} ^N g(t)^r\,t^{r/p-1} dt,
\end{equation*}
and, consequently,  since  $\int_0^\infty  g(t)^r\, t^{r/p-1}dt<\infty$,
\begin{equation}\lim_{N\to\infty} g(N)^r\, N^{r/p}=0.
\end{equation}

 Let now $\Psi(t)$ be defined by the equation
 \begin{equation}
\Psi\big( g(t)\big) =g(t)^r\,t^{r/p-1},
 \end{equation}
and let $\varphi(t)$ be given by 
 \begin{equation}
 \Psi(t)= t^r \varphi(t). 
 \end{equation}

From (16) and (17) it follows  that
\begin{equation} \varphi\big( g(t)\big)= t^{r/p-1}.
\end{equation}
This gives that $\varphi(t)$ increasing, and, consequently, $\Psi(t)$, and $\Psi(t)/t^r$, are increasing. 

Next we verify that $\Psi$ satisfies (3). Since  $\Psi(t)=t^r \varphi(t)$ and $r/p+r/q=1$, by (18) it follows that
\begin{equation} \int_0^\infty \frac{t^{q-1}}{\Psi(t)^{q/p}}\,dt = \int_0^\infty 
\frac{t^{q-1} t^{r-q}} {\varphi(t)^{q/p}}\,dt
=\int_0^\infty t^{r-1}  \varphi(t)^{1-q/r}\,dt\,.
\end{equation}

By the substitution $t=g(u)$, since $r/p -1 = -r/q$ and $1-q/r=-q/p$, the right-hand side of (19) becomes
\begin{equation} -\int_0^\infty g(u)^{r-1} \,\big( u^{r/p-1}\big)^{1-q/r} \, g'(u)\, du=-\int_0^\infty g(u)^{r-1}\,\, g'(u)\, u^{r/p}\,du\,.
\end{equation}

Now, on account of (14) and (15),  integration by parts gives that (20) evaluates to
\[ c \int_0^\infty g(u)^{r}\, u^{r/p-1}\,du, 
\]
which by (12) and (13) is finite, and (3) holds.

Moreover, by (16) it follows that
\[\int_0^\infty \Psi\big(g(u)\big)\, du<\infty,
\]
and, consequently, by (9) and (11), 
\begin{equation} \int_0^\infty \Psi\big(f(u)/2\big)\, du\le \int_0^\infty \Psi\big(g_0(u)/2\big)\, du\le \int_0^\infty \Psi\big(g(u)\big)\, du<\infty.
\end{equation}

Repeating the above argument with $2f$ replacing $f$ above, (21) becomes
\[\int_0^\infty \Psi\big(2f(u)/2\big)\, du = \int_0^\infty \Psi\big(f(u)\big)\, du<\infty ,
\]
(10) holds, and the proof is finished.\taf
\end{proof}

Since as noted in the proof $\Psi(t)/t$ increases when $r=1$,  $\Psi(t)$  can be regularized to a convex function $\Psi_0(t)$ such that $\Psi_0(t)\le \Psi(t) \le \Psi_0(2t)$, and, therefore, in this case the Orlicz type class $L_\Psi$ is essentially equivalent to an Orlicz space \cite{MJ, AT}.

\end{document}